\documentclass[10pt, article]{amsart}
\usepackage{geometry} 
\usepackage{ae} 
\usepackage[T1]{fontenc}
\usepackage[cp1250]{inputenc}
\usepackage{amsmath}
\usepackage{amssymb, amsfonts,amscd,verbatim}
\usepackage{mathtools}
\usepackage{MnSymbol}
\usepackage{tikz}
\usetikzlibrary{positioning}
\usepackage{graphicx}
\usepackage{tikz-cd}
\usepackage{comment}

\usepackage[normalem]{ulem}
\usepackage{indentfirst}
\usepackage{latexsym}
\input xy
\xyoption{all}

\usepackage[bookmarks=true,
              bookmarksnumbered=true, breaklinks=true,
              pdfstartview=FitH, hyperfigures=false,
              plainpages=false, naturalnames=true,
              colorlinks=true,pagebackref=false,
              pdfpagelabels]{hyperref}

\usepackage{amsmath}    
\newcommand{\ignore}[1]{ }

\newcommand{\pDn}{\mathcal{D}_{\mathcal{W},p}^n} %
\newcommand{\pDi}{\mathcal{D}_{\mathcal{W},p}^{< \infty}} %
\newcommand{\iDn}{\mathcal{D}_{\mathcal{B}}^n} %
\newcommand{\iDi}{\mathcal{D}_{\mathcal{B}}^{< \infty}} %
\newcommand{\D}[2]{\mathcal{D}_{{#1}}^{{#2}}} %

\theoremstyle{plain}
\newtheorem{Pocz}{Poczatek}[section]
\newtheorem{Proposition}[Pocz]{Proposition}

\newtheorem{Theorem}[Pocz]{Theorem}
\newtheorem{Corollary}[Pocz]{Corollary}

\newtheorem{Lemma}[Pocz]{Lemma}

\theoremstyle{definition}
\newtheorem{Definition}[Pocz]{Definition}
\newtheorem{Remark}[Pocz]{Remark}

\newtheorem*{theorem*}{Theorem}

\DeclareMathOperator*{\diam}{diam}

\def\RR{{\mathbb R}}

\def\ZZ{{\mathbb Z}}
\def\NN{{\mathbb N}}

\def\f{{\varphi}}

\def\UU{{\mathcal{U}}}
\def\VV{{\mathcal{V}}}
\def\WW{{\mathcal{W}}}
\def\BB{{\mathcal{B}}}

\def\asdim{\mathrm{asdim}}
\def\dim{\mathrm{dim}}

\def\dim{\mathrm{dim}}
\def\diam{\mathrm{diam}}

\numberwithin{equation}{section}

\author{Atish ~ Mitra}
\address{Montana Tech}
\email{atish.mitra@gmail.com}

\author{\v Ziga Virk}
\address{University of Ljubljana}
\email{ziga.virk@fri.uni-lj.si}

\title[The Space of Persistence Diagrams on $n$ Points Coarsely Embeds into Hilbert Space]%
  {The Space of Persistence Diagrams on $n$ Points Coarsely Embeds into Hilbert Space}
\thanks{This research  was partially supported by a bilateral grant BI-US/18-20-060 of ARRS. The first named author wishes to thank the Institute of Mathematics, Physics and Mechanics, and the University of Ljubljana, for their hospitality during his visit. The second named author was supported by Slovenian Research Agency grants N1-0114, P1-0292, J1-8131, and N1-0064.
The authors would like to thank the referee for valuable comments.}

\date{ \today
}


\includecomment{comment}
\begin{document}

\maketitle

\begin{abstract}
We prove that the space of persistence diagrams on $n$ points (with the bottleneck or a Wasserstein distance) coarsely embeds into Hilbert space by showing it is of asymptotic dimension $2n$. Such an embedding enables utilisation of Hilbert space techniques on the space of persistence diagrams. We also prove that when the number of points is not bounded, the corresponding spaces of persistence diagrams do not have finite asymptotic dimension. Furthermore, in the case of the bottleneck distance, the corresponding space does not coarsely embed into Hilbert space.
\end{abstract}


\section{Introduction}

Persistent homology is a version of homology encompassing  multiscale information about the underlying space. In the classical setting it produces a representation called the persistence diagram. This presentation has two important properties: it is planar (and hence visually easy to analyze) and stable with respect to the input, when an appropriate metric (the Bottleneck or the Wasserstein metrics) is used on the space of persistence diagrams.  (See \cite{Edels} for details) These two properties have played an important role in recent development of persistence in both applied and theoretical setting.  However,  the tools of statistics and machine learning usually rely on the structure of a Hilbert space, so a  question about the embedding of spaces of persistence diagrams arises naturally from applied perspective. Any embedding of this sort would provide an important link between topological data analysis and statistical tools.

Such embeddings have been considered before with mostly negative results. Roughly speaking, for certain spaces of persistence diagrams  there are no isometric \cite{Turner}, bilipshitz \cite{Bauer} or coarse \cite{Bub} (see Remark \ref{RemBub})  embeddings  into a Hilbert space (for a precise statements consult the mentioned papers). 
Basic properties of spaces of persistence diagrams (and why they are not a Hilbert space) have been established in a number of papers including \cite{Mil}, \cite{Bub0}. 

In this paper we consider coarse embeddings (i.e., approximate embeddings with a bound on the size of discontinuities) of certain spaces of persistence diagrams into the Hilbert space. The above-mentioned results suggest that positive embedding results are to be expected for coarse embeddings above all. The ideas of coarse geometry (also called asymptotic topology) were originally motivated by geometric group theory and works of Gromov \cite{Gro}. The field itself became even more active after Yu \cite{Yu} showed that finite asymptotic dimension, and more generally coarse embeddings into Hilbert space, provide sufficient conditions in the context of the Novikov conjecture. Consequently, a broad study of coarse embeddings and their connection with coarse properties was initiated. 

The main results of this paper are the following:
\begin{description}
 \item[Theorem \ref{asdim eq 2n}] The space of persistence diagrams on $n$ points (with any of the mentioned metrics) is of  asymptotic dimension  $2n$ and hence coarsely embeds into  Hilbert space. 
 
This is the first positive result about embeddings of persistence diagrams.
   
 \item[Theorem \ref{ThmCEH}] The space of persistence diagrams on finitely many points equipped with the bottleneck distance does not coarsely embed into  Hilbert space. 
\end{description}
A notable technical contribution of this paper is a reformulation of metrics on the spaces of persistence diagrams in Section \ref{SectP}. While somehow deviating from the Euclidean intuition, the reformulation provides a shorter definition of metrics and an efficient use \ of tools of coarse geometry, often leading to shorter proofs as (for example) in Section \ref{SectAsdim}. See Remarks \ref{RemDef} and \ref{RemDefFunBottle} for details. When the number of points in persistence diagrams is not bounded, it is easy to see that the underlying space of persistence diagrams is not of finite asymptotic dimension (Corollary \ref{CorAsdimi}).

\textbf{Related work:}  There are many maps (sometimes called kernels) from a space of persistence diagrams to a Hilbert space that are in use today. Some of them are listed  in \cite{Bauer}, where an analysis of bilipschitz embedding is performed. \cite{Bell} shows that the space of persistent diagrams on finitely many points fails to have property A (hence is not of finite asymptotic dimension) in the Wasserstein metrics. A result of \cite{Bub} is closely related to Theorem \ref{ThmCEH} (see Remark \ref{RemBub} for details) and shows that the space of persistent diagrams on countably many points in the bottleneck metric does not coarsely embed into Hilbert space. A further work of Wagner \cite{Wag} shows that the same space equipped with the Wasserstein metric for $p>2$ does not  coarsely embed into Hilbert space either. Computing the asymptotic dimension we rely on a result of \cite{Kas} (Theorem \ref{Kasp}) stating that finite group actions preserve asymptotic dimension. This action is a particular example of coarsely $n$-to-$1$ maps, for which the asymptotic dimension can be controlled \cite{KZ}.  It seems that our computation of asymptotic dimension can be adjusted to simplify the computation in hyperspaces and provide a short proof of \cite{Shu}.

\textbf{Structure of the paper:}
In Section \ref{SectP} we introduce basics on persistence diagrams (including an unorthodox definition) and coarse geometry. In Section \ref{SectAsdim} we calculate the asymptotic dimension of persistence diagrams on $n$ points. In Section \ref{SectNonEmbd} we consider embeddings of arbitrary finite metric spaces into the space of persistence diagrams on finitely many points (with the bottleneck distance) and prove that the later does not coarsely embed into Hilbert space. 

\textbf{Note on this version:}
In an earlier version of the paper, we had claimed  that a direct application  of Kasprowski's result from \cite{Kas} on an appropriately constructed space shows that the asymptotic dimension of the space of persistence diagrams on $n$ points is $2n$. However,  as the space constructed by us was not proper, Kasprowski's theorem (which needs the space to be proper) could not be directly used there. However, our result (Theorem \ref{asdim eq 2n}) still holds, and we have corrected the proof in this version.

\section{Preliminaries}
\label{SectP}
In this section we introduce notation and technical preliminaries required for our results.

\subsection{Persistence Diagrams}

Persistence diagrams appear as planar visualisations of persistence modules and persistent homology. We will first build up a notation that encompasses most of the interesting cases of spaces of persistent diagrams.

\begin{Definition}\label{Def1}
Introducing preliminary setting we define:
\begin{enumerate}
 \item  metric $d_\infty$ on $\RR^2$ by $d_\infty((x_1, x_2),(y_1,y_2))=\max \{|x_1-y_1|,|x_2-y_2|\}$;
 \item $\D{}{1}=T \cup \{\Delta\} $ where $\Delta \notin T=\{(x_1,x_2)\in \RR^2 \mid x_2 > x_1 \geq 0\} $;
 \item semi-metric $\delta$ on $\D{}{1}$ as an extension of $d_\infty|_{T}$ on $T$ by defining $\delta((x_1,x_2),\Delta)= (x_2-x_1)/2$.
\end{enumerate}
\end{Definition}

\begin{Remark}
 In Definition \ref{Def1} we could replace $T$ by $\RR^2$ or $\{(x_1,x_2)\in \RR^2 \mid x_2 > x_1\} $. We choose to opt for the current definition as it seems more standard. All results and arguments mentioned in this paper hold for the other cases as well.
 
 Point $\Delta$ represents the diagonal $\{(x_1,x_2)\in \RR^2 \mid x_2 = x_1\} $ in the usual description of persistence diagrams and $\delta((x_1,x_2),\Delta)$ is actually the $d_\infty$ distance from $(x_1,x_2)$ to the diagonal. We find it technically easier to do analysis of the spaces of persistence diagrams by considering the whole diagonal as one point rather than a collection of infinitely many points, as is usually done in the literature.
\end{Remark}

\begin{Definition}
 \label{Def2}
 Choose $n\in \NN$.  Introducing spaces of persistence diagrams we define:
\begin{enumerate}
\item matching (pairing) to be a bijection between sets. If the sets are the same then the matching is a permutation;
 \item \textbf{the space of persistence diagrams} on at most $n$ points as $\D{}{n}= (\D{}{1})^n/{\mathcal{S}_n}$, where the group of symmetries $\mathcal{S}_n$ acts on the coordinates by permutation, i.e., we identify diagrams $z=(z_1, z_2, \ldots,z_n), z'=(z'_1, z'_2, \ldots,z'_n)\in  (\D{}{1})^n$ iff there exists a matching $\f$ on $\{1,2,\ldots,n\}$ so that $z_i = z'_{\f(i)}$;
 \item  a natural inclusion $\D{}{n} \subset \D{}{n+1}$ by appending point $\Delta$. We will frequently use this inclusion implicitly, for example by identifying diagrams $(a)$ and $(a,\Delta)$. Consequently we can define 
 $\D{}{< \infty}=\bigcup_{n\in \NN}{\D{}{n}}$.
\end{enumerate}
\end{Definition}

\begin{Definition}
 \label{Def3}
Let $n\in \NN$ and $p>1$.
Introducing metrics on the spaces of persistence diagrams we define:
\begin{enumerate}
 \item \textbf{bottleneck distance} $d_{\mathcal{B}}$ on $\D{}{n}$:  for points $z=(z_1, z_2, \ldots,z_n), z'=(z'_1, z'_2, \ldots,z'_n)$ in $\D{}{n}$ define 
 $$
 A(z)=(z_1, z_2, \ldots, z_n, \Delta, \ldots, \Delta)\in  (\D{}{1})^{2n}
 $$
 $$
  A(z')=(z'_1, z'_2, \ldots, z'_n, \Delta, \ldots, \Delta)\in  (\D{}{1})^{2n}
 $$
 i.e., we append $n$ copies of $\Delta$ to each of the diagrams by defining $z_i=z'_i=\Delta, \forall i\in \{n+1, n+2, \ldots, 2n\}$.  The bottleneck distance is defined as 
 $$
 d_{\mathcal{B}}(z,z') = \min_
 {\f \in  {\mathcal{S}_{2n}}} \max_i \delta (z_i, z'_{\f(i)}). 
 $$
 Matching $\f$, for which the minimum above is obtained, is called optimal. See Remark \ref{RemDef} for technical clarifications.
 \item $\iDn=(\D{}{n},d_\mathcal{B})$ and $\iDi=(\D{}{<\infty},d_\mathcal{B})$. Note that $\D{\mathcal{B}}{1}$ is not isometric to $(\D{}{1},\delta)$.
 \item $p$-\textbf{Wasserstein distance} $d_{\mathcal{W},p}$ on $\D{}{n}$:  for points $z=(z_1, z_2, \ldots,z_n), z'=(z'_1, z'_2, \ldots,z'_n)$ in $\D{}{n}$  use the notation of (2) to define  $$
d_{\mathcal{W},p}(z,z') = \min_
 {\f \in  {\mathcal{S}_{2n} }} \Big( \sum_i \delta (z_i, z'_{\f(i)})^p\Big) ^{1/p}. 
 $$
 Matching $\f$, for which the infimum above is obtained, is called optimal. 
 \item $\pDn=(\D{}{n},d_p)$ and $\pDi=(\D{}{<\infty},d_p)$.\end{enumerate}
 Note that the convention of Definition \ref{Def2}(3) implies that metrics $d_{\mathcal{B}},$ and $d_{\mathcal{W},p}$ are well defined even if $z\in \D{}{n}$ and $z'\in \D{}{m}$ for $n \neq m$. For example, if $n>m$ we append  $m-n$ copies of $\Delta$ to $z'$ to compute the mentioned distances.
\end{Definition}

Matching on $\{1,2,\ldots,n\} $ is perfect for  $z=(z_1, z_2, \ldots,z_n), z'=(z'_1, z'_2, \ldots,z'_n)$ in $\D{}{n}$ if the following holds: $z_i = \Delta$ iff $z'_{\f(i)}=\Delta$.

\begin{Remark}
 \label{RemDef}
 In this remark we provide clarifications to Definition \ref{Def3}. The  distances in (1) and (3) in definition \ref{Def3} are defined in a non-standard way. Rather than adding infinitely many points on the diagonal, we append to each diagram with $n$ points only $n$ copies of the diagonal point, which suffice to accommodate the usual matching with diagonal points. It is clear that this definition is a restatement (using action of $S_{2n}$) of the usual definition \cite{Edels} of  distance using partial matching between persistence diagrams, where the unmatched points are matched to the closest distance point of the diagonal. These re-definitions of the distances have the advantage of using only one term instead of the usual  three terms and the inclusion of infinitely many diagonal points. 
\end{Remark}

The importance of the mentioned distances on  spaces of persistence diagrams arises from Stability results (see  \cite{Edels} for an overview of those), which state that in a certain sense, the persistent diagrams vary continuously with respect to the underlying filtrations or datasets.


For the proof of Theorem \ref{asdim eq 2n}, we find it convenient to use an alternate but equivalent way of defining the bottleneck metric. We describe it below.

\begin{Definition}
 \label{DefFunBottle} 
 We first define the bottleneck distance on the space of persistence diagrams on one point, and use it to define the bottleneck distance on the space of persistence diagrams on $n$ points:
 
 \begin{enumerate}
 \item  (\textbf{Bottleneck distance} $d^1_{\mathcal{B}}$ on $\D{}{1}$)
 
  for points $z_1,  z'_1$ in $\D{}{1}$ , the bottleneck distance is defined as 
$$
 \displaystyle{d^1_{\mathcal{B}}(z,z') = \min \{d_\infty(z,z'), \max \{ \delta(z,\Delta), \delta(z',\Delta)\}\}}. 
$$

 \item  (\textbf{Bottleneck distance} $\tilde{d}_{\mathcal{B}}$ on $\D{}{n}$) 
 
 for points $z=(z_1, z_2, \ldots,z_n), z'=(z'_1, z'_2, \ldots,z'_n)$ in $\D{}{n}$, bottleneck distance is defined as 
 $$
 \displaystyle{\tilde{d}_{\mathcal{B}}(z,z') =  \min_{\f \in  {\mathcal{S}_{n}}} \max_{i} \{  d^1_{\mathcal{B}}(z_i,z_{\f(i)}') }. 
 $$

 \end{enumerate}
  \end{Definition}

\begin{Remark}
 \label{RemDefFunBottle}
 In this remark we provide clarifications to Definition \ref{DefFunBottle}.  First we note that the bottleneck distance defined on the space of persistence diagrams on one point matches the definition in \ref{Def3}. To see that  $d_{\mathcal{B}}$ and  $\tilde{d}_{\mathcal{B}}$ define the same distance, we consider the usual definition of  bottleneck distance using partial matchings. As any permutation of the set $\{1,2, \cdots, n\}$ trivially describes a partial map between the two $n$-point persistence diagrams, we only have to show that $\tilde{d}_{\mathcal{B}} \le d_{\mathcal{B}}$ for all partial matches between two $n$-point persistence diagrams. To that end, suppose that the distance between two $n$-point persistence diagrams is   realized by $d_{\mathcal{B}}$ by an optimal partial matching of the two persistence diagrams. By the definition, this realized distance arises either as $d_\infty(z_i, z'_j)$, or as $ \delta(z_i,\Delta)$, or as $ \delta(z'_j,\Delta)$ for some $i,j$. If the distance is realized by $d_\infty(z_i, z'_j)$  there is some $\f \in S_n$ extending the partial matching (see Remark \ref{RemDef}) such that $\f (i)=j$, which (as it comes from an optimal partial match) implies that the distance is $d^1_{\mathcal{B}}(z_i,z'_{\f(i)})$. Once again, as this distance comes from an optimal partial match, this gives us $\tilde{d}_{\mathcal{B}} \le d_{\mathcal{B}}$. If the distance is realized by $ \delta(z_i,\Delta)$ or by $ \delta(z'_j,\Delta)$, a similar argument shows again $\tilde{d}_{\mathcal{B}} \le d_{\mathcal{B}}$.

We note that definition  \ref{DefFunBottle}   allows us to  express the space of persistence diagrams on $n$-points as a natural quotient of the $n$-fold product of the metric space $(\D{}{1},\tilde{d}_\mathcal{B})$ by the finite group $S_n$, which is crucial in  our computation of asymptotic dimension in Theorem \ref{asdim eq 2n}.

\end{Remark}


\subsection{Coarse Geometry}

We will now introduce the basic terms and definitions of coarse geometry that are used in this paper. The first concepts we introduce are  the notions of coarse embedding and coarse equivalence.

\begin{Definition}
 \label{Def4}
Let   $f:X \to Y$ be a function between metric spaces.

\begin{enumerate}
\item $f$ is  is said to be a coarse embedding if for $i=1,2$ there are non-decreasing functions $\rho_i:[0,\infty) \to [0, \infty)$  with $\rho_1(d(x_1,x_2)) \le d(f(x_1),f(x_2)) \le \rho_2(d(x_1,x_2)) $ and with $\lim_{t \to \infty} \rho_1(t) = \infty$.
\item If, in addition, $f$ is coarsely onto  then $f$ is said to be a coarse equivalence. A function $f: X \to Y$ is said to be coarsely onto if there is $D > 0$ such that the $D$-neighborhood of $f(X)$ is all of $Y$ (for every $y \in Y$ there is $x \in X$ such that $d(f(x),y) \le D$). 
\end{enumerate}

\end{Definition}

Next we introduce the concept of asymptotic dimension, which turns out to be the appropriate concept of dimension in coarse geometry.

\begin{Definition}
 \label{Def5}
Let $n$ be a non-negative integer. We say that the asymptotic dimension of a metric space $X$ is less than or equal to $n$ ($\asdim X \le n$) iff for every $R > 0$ the space $X$ can be expressed as the union of $n+1$ subsets $X_i$, with each $X_i$ being an union of uniformly bounded $R$-disjoint sets.

\end{Definition}

Asymptotic dimension is a coarse invariant, i.e., coarsely equivalent spaces have the same asymptotic dimension. For a self contained survey of asymptotic dimension see \cite{BD}.

As an example, to see that $\asdim (\mathbb{R}) \le 1$ we need to  (for each $R >0$) express $\mathbb{R}$ as the union of two (2) families of uniformly  bounded $R$-disjoint sets. See Figure \ref{FigAtish} for a decomposition of $\mathbb{R}$ into two such families. One can use similar "brick decompositions" to get upper bounds of asymptotic dimension of $\mathbb{R}^n$ (for any $n$). In general, asymptotic dimension behaves as expected in terms of unions and products as the following statement demonstrates.

\begin{figure}
\begin{tikzpicture}[scale=.9]

\draw[very thick] (-2,0)--(-4,0);
\draw (0,0)--(-2,0);
\draw[very thick] (0,0)--(2,0);
\draw (2,0)--(4,0);
\draw[very thick] (4,0)--(6,0);
\draw (6,0)--(8,0);
\node at (8.5,0) {$\ldots$};
\node at (-4.5,0) {$\ldots$};
\end{tikzpicture}
\caption{$\asdim \mathbb{R} \le 1$}
\label{FigAtish}
\end{figure}

\begin{Theorem}
\label{ThmUP}
Suppose $X$ and $Y$ are subspaces of a  metric space $Z$. Then the following hold:
\begin{description}
 \item[Union Theorem] $\asdim (X \cup Y) = \max\{\asdim X, \asdim Y\}$ \cite[Corollary 26]{BD}.
  \item[Product Theorem] $\asdim (X \times Y) \leq \asdim X + \asdim Y$ \cite[Theorem 32]{BD}.
\end{description}
\end{Theorem}

Getting lower bounds of asymptotic dimension usually need special techniques such as homological methods. However, here we will use Lemma \ref{embedlargecubes} as a direct way of getting lower bounds on asymptotic dimension of a space. As this lemma uses the notion of topological dimension, we give a definition of topological dimension below for completeness (for more details see \cite{Engel},\cite{BD}). We recall that the multiplicity of a cover of a metric space is the maximum number of elements of the cover that can intersect. The second characterization of topological dimension given below is usually called metric dimension, which coincides with topological dimension for compact metric spaces. Given $\epsilon>0$ we say that a collection of subsets of a metric space is $\epsilon$-small, if the diameter of each of the sets is at most $\epsilon$.

\begin{Definition}
 \label{Def5.5}
\text{}Definitions of topological dimension (see \cite{Engel} for details):
  \begin{enumerate}
 \item Let $n$ be a non-negative integer. We say that the topological  dimension of a topological  space $X$ is less than or equal to $n$ ($\dim X \le n$) iff for every open cover $\UU$ of the space $X$ there is an open cover $\VV$ of $X$ of multiplicity less than or equal to $n+1$.
 \item If $X$ is a compact metric space, the above definition is equivalent to the following: $\dim X \le n$ iff for every $\epsilon>0$, $X$ has an  $\epsilon$-small open cover of multiplicity $n+1$. Another equivalent definition is the following:  $\dim X \le n$ iff for every $\epsilon>0$,  $X$ has an open  cover consisting of $n+1$ families, such that each family consists of disjoint $\epsilon$-small open sets.
 \end{enumerate}
\end{Definition}

Throughout the paper we will use various metrics on the product of spaces: if $(X,d)$ is a metric space, $n\in \NN$, and $p \ge 1$, we can  define metrics on 
 $X^n$ by 
$$
d_\infty (z,z') =\max_{i} d (z_i,z'_i), \qquad d_p (z,z') =\Big(\sum_{i=1}^{n} d (z_i,z'_i)^p\Big) ^{1/p}
$$
for for points $z=(z_1, z_2, \ldots, z_n), z'=(z'_1, z'_2, \ldots, z'_n) \in X^n$.
We will often refer to the $d_\infty$ metric as the $\max$ metric.

\begin{Lemma}\label{embedlargecubes}
Let $p>1$. If for every $R>0$ there is an isometric embedding of $([0,R]^m,d_\infty)$ or $([0,R]^m,d_p)$ in $X$, then   $\asdim X \ge m$.
\end{Lemma}

\begin{proof}
Assuming (towards a contradiction) that $\asdim X \le m-1$, we can get  a cover $\mathcal U$ of $X$ by $m$  families of uniformly $R$-bounded $1$-disjoint sets, for some $R>0$. By hypothesis, given any $1>\epsilon>0$, space $[0, \frac{R+1}{\epsilon}]^m$ (with either the $d_p$ or $d_\infty$ metric) isometrically embeds in $X$. Restricting $\mathcal U$ to (an isometrically embedded)  $[0, \frac{R+1}{\epsilon}]^m$ in $X$ we obtain a cover of $[0, \frac{R+1}{\epsilon}]^m$ by $m$  families of uniformly $R$-bounded $1$-disjoint sets. Replacing each set of this cover by its open neighborhood of radius $1/2$ and maintaining the structure of the cover,  we obtain a cover of $[0, \frac{R+1}{\epsilon}]^m$ by $m$  families of uniformly $(R+1)$-bounded disjoint open sets. Scaling $[0, \frac{R+1}{\epsilon}]^m$ and the obtained cover with a scale factor of $\frac{\epsilon}{R+1}$,  we get an $\epsilon$-small open cover of $[0,1]^m$ consisting of $m$ families of disjoint open sets. This contradicts the fact that  the topological dimension of $[0,1]^m$ is $m$ in $d_p$ and $d_\infty$. \end{proof}

\begin{Corollary}
 \label{CorAsdimi} For each $p>1$ spaces $\iDi$ and $\pDi$ are not of finite asymptotic dimension.
\end{Corollary}

\begin{proof}
For each $R>0$ and $n\in \NN$ we can isometrically embed $([0,R]^n, d_\infty)$ or $([0,R]^n, d_p)$  into $\D{<\infty}{n}$ or $\D{p}{<\infty}$ respectively by mapping $(x_1, x_2, \ldots, x_n)\mapsto (2R, 4R +x_1, 4R, 6R + x_2, \ldots, 2nR, 2nR + 2R + x_n)$. The conclusion follows by Lemma 
 \ref{embedlargecubes}.
\end{proof}

Finiteness of asymptotic dimension is closely related to embeddability questions - as the following well known result shows.

\begin{Theorem} \label{FinasdimImpliesCEH}
 [Roe, \cite{Roe} Example 11.5]
A metric space of finite asymptotic dimension coarsely embeds in Hilbert space.
\end{Theorem}

 As an example of the importance of Hilbert space  embeddability in coarse geometry, see \cite{Yu}. One of our interests in the current paper is questions of embeddability: whether there there exist embeddings of interesting spaces in the spaces of persistence diagrams, and whether the spaces of persistence diagrams themselves can be embedded in interesting spaces.

In Section \ref{SectAsdim} we will use the following result about behavior of asymptotic dimension under finite group actions, to get an exact value of $\asdim \iDn$. When a finite group $F$ acts by isometries on a metric space $X$, we will define the metric on $X/F$ by $d_F(Fx,Fx')=\min_{f \in F} d(x,fx')$. 

\begin{Theorem} \label{Kasp}
 [Kasprowski, \cite{Kas} Theorem 1.1]
Let $X$ be a proper metric space and $F$ be a finite group acting on $X$ by isometries. Then $X/F$ has the same asymptotic dimension as that of $X$.
\end{Theorem}


 Given a sequence of bounded metric spaces $(X_n,d_n)$ we can define a metric $d$ on their  disjoint union $\bigsqcup_n X_n$ such that $d$ restricted to $X_n$ is $d_n$, and for $i \ne j$ and $x_i \in X_i$, $x_j \in X_j$, $d(x_i,x_j)> \max\{\text{diam}(X_i), \text{diam}(X_j)\}$. Any two such metrics on  $\bigsqcup_n X_n$ are coarsely equivalent and the resulting space is called \textbf{coarse disjoint union} (it appears, for example, in Theorem  \ref{ThmDra}).

In Theorem  \ref{ThmDra} below we will consider $\ZZ_k=\{[0], \cdots, [k-1]\}$, the set of integers modulo $k \in \NN$, as a metric space. The metric is defined as $d([i],[j])=\min\{|i'-j'|: [i'-j']=[i-j]\}$. This is the usual word metric on the finitely generated group $\ZZ_k$.

As mentioned above, the question of coarse embeddability (and non-embeddability) of metric  spaces into Hilbert space have been studied extensively. In Section \ref{SectNonEmbd} we will use the following result to show that $\iDi$ does not coarsely embed in a Hilbert space.

\begin{Theorem} \label{ThmDra}
 [Dranishnikov et al, \cite{DranGLY} Proposition 6.3]
 
Consider $(\mathbb{Z}_n)^m$ as a metric space, where  the integers mod $n$ has the word metric and the $m$-fold product has the max metric $d_{\infty}$. Let $S$ be the disjoint union of $(\mathbb{Z}_n)^m$ (for all $m,n \ge 1$). We define a metric $d$ on $S$ whose restriction to each $(\mathbb{Z}_n)^m$ coincides with its existing metric, and such that  $d(x,y) > m+n+m'+n'$ for $x \in (\mathbb{Z}_n)^m$ and $y \in (\mathbb{Z}_{n'})^{m'}$. Then $S$ does not coarsely embed in a Hilbert space.

\end{Theorem}

Often, an efficient way to decide coarse embeddability (and non-embeddability) of metric  spaces is the following result, which says that this question is "finitely determined".

\begin{Theorem} \label{ThmNow}
 [Nowak, \cite{Now} Theorem 3.4]
 
A metric space $(X,d)$ admits a coarse embedding in a Hilbert space if and only if for $i=1,2$ there are non-decreasing functions $\rho_i:[0,\infty) \to [0, \infty)$  with  $\lim_{t \to \infty} \rho_1(t) = \infty$, such that for every finite subset $A \subset X$ there exists a map $f_A: A \to \ell_2$ satisfying 
$\rho_1(d(x_1,x_2)) \le \lVert f_A(x_1)-f_A(x_2) \rVert_2 \le \rho_2(d(x_1,x_2)) $ for all $x_1,x_2 \in X$.

\end{Theorem}


\section{Asymptotic Dimension of Spaces of Persistence Diagrams with at most n points}
\label{SectAsdim}

In this section we compute the exact  asymptotic dimension of the space of persistence diagrams with at most $n$ points with either the bottleneck distance ($\iDn$) or the p-Wasserstein distances ($\pDn$). Due to the following Proposition \ref{n-point diags coarsely eqv}, it suffices to prove the result just for the case of  the bottleneck distance.
By \ref{FinasdimImpliesCEH}, the finiteness of asymptotic dimension of these spaces imply that they admit coarse embeddings into Hilbert space.

\begin{Proposition}\label{n-point diags coarsely eqv}
For each $n \in \mathbb{N}$ and $p \ge 1$, $\iDn$ and $\pDn$ are coarsely equivalent.
\end{Proposition}

\begin{proof}
This can be checked by direct comparison of the definitions of these metrics.
\end{proof}

The main result of this section is the following.

\begin{Theorem}\label{asdim eq 2n}
For $n \in \mathbb{N}$, $\asdim \pDn =\asdim \iDn  =2n$.
\end{Theorem}

The following lemma deals with the case $n=1$.

\begin{Lemma}\label{asdim D1infty is 2}
$\asdim \D{\mathcal{B}}{1} = 2$
\end{Lemma}

\begin{proof}
The inequality $\asdim \D{\mathcal{B}}{1} \le 2$ is obtained by adapting the usual 3-colored brick decomposition of the plane (see \cite{BD}) to our case. Choose $R>0$ as in Definition \ref{Def5}. The decomposition we will be using is depicted in Figure \ref{brick}. It consists of rectangles of sides $2R \times R$ and a large monochromatic (grey) set $B$  neighboring the diagonal. Note that the distance between any pair of rectangles of the same color (excluding the neighboring pairs that form $B$) is at least $R$. The rectangles are of diameter  $2R$. Note that  $B$ is of diameter $10R$ as the horizontal distance from any point of $B$ to the diagonal is at most $2R + 2R + R = 5R$ (two and a half bricks), implying that $B$ is contained in the closed ball of radius $5R$ around the diagonal. Grouping the rectangles and $B$ by color we obtain a cover of $\D{\mathcal{B}}{1}$ by three subsets (white, striped and grey), each of which  consists of sets (bricks and, in the case of the grey subset,  $B$), which are $5R$-bounded and $R$-disjoint. Hence $\asdim \D{\mathcal{B}}{1} \le 2$ by Definition \ref{Def5}.

We now turn attention to inequality $\asdim \D{\mathcal{B}}{1} \geq 2$, which we prove using Lemma \ref{embedlargecubes}.  Given $R>0$ the subset $\widetilde B=[0,R]\times [2R,3R]$ in $\D{\infty}{1}$ is isometric to $([0,R]^2,d_\infty)$. To verify this note that $\widetilde B$ is of diameter $R$ and at distance $2R$ from the diagonal, hence no optimal matching used when computing the induced distance on $\widetilde B$ (recall Definition \ref{Def3}) pairs any point of $\widetilde B$ to the diagonal. The proof now follows by Lemma \ref{embedlargecubes}.
\begin{figure}
    \centering
    \includegraphics[width=4cm]{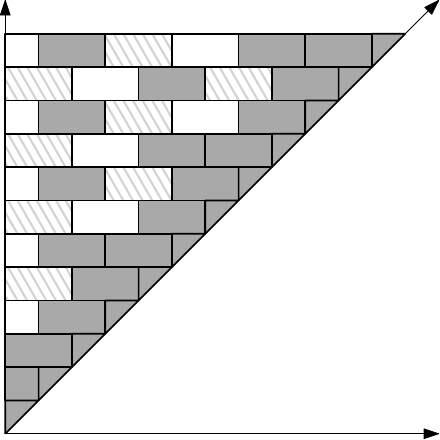}
    \caption{$\asdim \D{\mathcal{B}}{1} \le 2$}
    \label{brick}
\end{figure}
\end{proof}






Given families $\UU$ and $\VV$ of subsets of a metric space and $r>0$, we define the $r$-saturated union of $\UU$ and $\VV$ as
$$
\UU \cup_r \VV = \{N_r(U, \VV) \mid U \in \UU  \} \cup \{V \in \VV \mid d(V, U)>r \ \forall U\in \UU \}
$$
where 
$$
N_r(U, \VV) = U \cup \bigcup_{V\in \VV, d(U,V)\leq r} V.
$$
In this context we define $d(U,V)=\inf_{u\in U, v\in V} d(u,v)$.

\begin{Lemma}
 \cite[Proposition 24]{BD}
 \label{SashaCovers}
 Let $\UU$ be an $r-$disjoint, $R$-bounded family of subsets of $X$ with $R \geq r$. Let $\VV$ be a $5R$-disjoint, $D$-bounded family of subsets of $X$. Then $\VV \cup_r \UU$ is $r$-disjoint and $(D + 2(r + R))$-bounded.
\end{Lemma}

We now  prove Theorem \ref{asdim eq 2n}.

\begin{proof}
We want to prove that for each $r>0$ there exist uniformly bounded families $\UU_0, \UU_1, \ldots, \UU_{2n}$ consisting of $r$-disjoint subsets of $X$, whose union is a cover of $X$.

We will proceed by induction on $n$. Assume  $\asdim \mathcal{D}_{n-1}^\BB  =2(n-1)$, with Lemma \ref{asdim D1infty is 2} providing the initial case.
\begin{enumerate}
 \item 
 Decompose 
 $$
 \iDn = \widetilde \iDn \cup N(\mathcal{D}_{n-1}^\BB, r),
 $$
 where $\widetilde \iDn \subset \iDn$ is the collection of all diagrams, whose all $n$ points are at $d_\infty$ distance more than $r$ from the diagonal $\Delta$, and $N(\mathcal{D}_{n-1}^\BB, r)$ is the closed $r$ neighborhood of $\mathcal{D}_{n-1}^\BB \subset \iDn$ in $\BB$. Observe that if a point of a persistence diagram is at distance at most $r$ from $\Delta$, the diagram is at distance at most $r$ from a diagram in $\mathcal{D}_{n-1}^\BB$ obtained by replacing the mentioned point by $\Delta$, and thus $ \widetilde \iDn \cup N(\mathcal{D}_{n-1}^\BB, r)$ indeed equals $\iDn$.
 \item \label{App3}
 By the induction hypothesis $\asdim N(\mathcal{D}_{n-1}^\BB, r)= 2(n-1)$ as $\mathcal{D}_{n-1}^\BB$ is obviously coarsely dense in $N(\mathcal{D}_{n-1}^\BB, r)$.
 \item 
 Define 
 $$
 Z=\{(x,y)\in \RR^2, \mid y - r > x \geq 0\}.
 $$
\item 
\label{App1} It is elementary to observe that $(Z^n,d_\infty)$ is a subset of $(\RR^{2n},d_\infty)$ and that $\asdim (Z^n,d_\infty) = 2n$ by the monotonicity of asymptotic dimension, the fact that $\asdim (\RR^{2n},d_\infty)=2n$ (jointly implying $\asdim (Z^n,d_\infty) \leq 2n$) and Lemma \ref{embedlargecubes}.
\item 
Let $S_n$ act on $Z^n$ by permutation of components. The resulting quotient metric on $(Z^n, d_\infty)/S_n$ is
 $$
{d_\infty/S_n}(z,z') = \min_
 {\f \in  {\mathcal{S}_{n}}} \max_i d_\infty (z_i, z'_{\f(i)})
 $$
 for $z=(z_1, z_2, \ldots,z_n), z'=(z'_1, z'_2, \ldots,z'_n)\in Z^n$.
 \item As $(Z^n, d_\infty)$ is proper, so is $(Z^n, d_\infty)/S_n$ and the main result of \cite{Kas} combined with (\ref{App1}) implies $\asdim (Z^n, d_\infty)/S_n = 2n$.
 \item \label{App2}
 Choose  uniformly bounded families $\WW_0, \WW_1, \ldots, \WW_{2n}$ consisting of $r$-disjoint subsets of $Z_n/{S_n}$, whose union is a cover of $Z_n/{S_n}$. Choose $D>0$ as an upper bound on the diameter of sets of all $\WW_i$.
 \item We claim that families $\WW_0, \WW_1, \ldots, \WW_{2n}$ are uniformly bounded and $r$-disjoint in $\widetilde \iDn$ as well:
	\begin{enumerate}
 		\item First note that as sets $\widetilde \iDn$ and $Z_n/{S_n}$ are the same, the difference is in the metrics: $d_\BB$ and $d_\infty/S_n$.
		\item By the definition $d_\BB \leq d_\infty/S_n$ as the former may utilize matching with the diagonal as well, hence families $\WW_i$ are $D$-bounded in $d_\BB$ as well.
		\item If $d_\BB(x,y) < d_\infty/S_n (x,y)$ then the realizing matching in $d_\BB$ includes a matching to the diagonal. As the diagonal is at $d_\infty$-distance more than $r$ from all points of $x$ and $y$, we have $d_\BB(x,y)>r$. In particular, the disjointness may not decrease below $r$.
	\end{enumerate}
 \item By (\ref{App3}) we can also choose uniformly bounded families $\VV_0, \VV_1, \ldots, \VV_{2n}$ consisting of $5D$-disjoint subsets of $N(\mathcal{D}_{n-1}^\BB, r)$, whose union is a cover of $N(\mathcal{D}_{n-1}^\BB, r)$.
 \item By Lemma \ref{SashaCovers} families $\UU_0, \UU_1, \ldots, \UU_{2n}$ defined as $\UU_i = \WW_i \cup_r \VV_i$ are $r$-disjoint uniformly bounded and their union covers $\iDn$. 
\end{enumerate}
\end{proof}


\section{Non-Embeddability results}
\label{SectNonEmbd}

In this section we consider embeddings of finite metric spaces into $\iDn$ and prove that $\iDi$ does not coarsely embed into Hilbert space. For embeddings of separable bounded metric spaces see \cite[Theorem 3.1]{Bub}.

\begin{Lemma}\label{LemiDi}
For each $h\in \NN$ every finite metric space $(X,d)$ embeds isometrically into $\D{\mathcal{B}}{|X|}$ above the horizontal line at height $h$.
\end{Lemma}

\begin{proof}
 Let $X=\{x_0, x_1, \ldots, x_n\}$ and $R=\mathrm{diam}(X)$. For each $k$ define 
 $$
 f(x_k) = \Big\{\big(3Ri, 3Ri + 3R + d(x_k,x_i)+h\big) \mid i=1,2,\ldots,n\Big\}.
 $$
 Note that $f \colon X \to f(X)\subset \D{\mathcal{B}}{|X|}$ is an isometry due to the following facts: 
\begin{enumerate}
 \item for each $k$ the subset $f(x_k)$ consists of precisely one point at each $x$-coordinate of the form $3Ri, \forall i=1,2,\ldots,n$.
 \item for each $j$ and $k$, the optimal pairing between $f(x_j)$ and $f(x_k)$ is perfect  (no point is paired to the diagonal point) and always pairs points of the same $x$-coordinate.
 \item for each $j,k,i$, 
 $$
d_\infty\Big(\big(3Ri, 3Ri + 3R + d(x_j,x_i)+h\big), \big(3Ri, 3Ri + 3R + d(x_k,x_i)+h\big)\Big)= |d(x_j,x_i)-d(x_k,x_i)| \leq d(x_j,x_k)
 $$ 
 with the equality attained at $i=j$ and $i=k$.
\end{enumerate}
We conclude that for each $j,k$, $d(x_k,x_j)=d_\infty(f(x_j),f(x_k))$, hence $f$ is an isometry.
\end{proof}

\begin{Corollary}\label{CoriDi}
 A coarse disjoint union of any collection of finite metric spaces $\{A_i\}_{i\in \{1, 2, \ldots\}}$ embeds isometrically into $\iDi$. If for some $M\in \NN$ we have $|A_i|\leq M, \forall i$, then the embedded space lies within $\D{\mathcal{B}}{M}$.
\end{Corollary}

\begin{proof}
 Using Lemma \ref{LemiDi} we can isometrically embed each $A_i$ into $\D{\mathcal{B}}{|A_i|}$ above any height of our choosing. Starting with $A_1$ we inductively embed $A_i$  into $\iDi$ using Lemma \ref{LemiDi} so that the $y$-coordinates of the embedded $A_i$ are at least $\max_{j \leq i} \diam(A_j)$ above the maximal $y$-coordinates of the embedded $A_{i-1}$. Let $\widetilde A$ denote  the embedded union of $\{A_i\}_{i\in \{1, 2, \ldots\}}$. The subspace metric on $\widetilde A$ turns $\widetilde A$ into a coarse disjoint union of  $\{A_i\}_{i\in \{1, 2, \ldots\}}$. 
  
If for some $M\in \NN$ we have $|A_i|\leq M, \forall i$, then the embedding above maps each $A_i$ into $\D{\mathcal{B}}{M}$ by Lemma \ref{LemiDi}, and hence $\widetilde A \subset \D{\mathcal{B}}{M}$.
\end{proof}

\begin{Theorem}\label{ThmCEH}
 $\iDi$ does not coarsely embed into Hilbert space.
\end{Theorem}

\begin{proof}
Follows from Theorem \ref{ThmDra} and  Corollary \ref{CoriDi} for the coarse disjoint union of $((\mathbb{Z}/m)^n, d_\infty)$.
\end{proof}

In particular, the asymptotic dimension of $\iDi$ is not finite.

\begin{Remark}\label{RemBub}
During the completion of the first version of this manuscript a preprint \cite{Bub} was posted which independently presented similar arguments and proved that a space of persistence diagrams on countably many points equipped with the corresponding version of the bottleneck distance does not coarsely embed into Hilbert space. Since the space in \cite{Bub} naturally contains  $\iDi$, the result of \cite{Bub} follows from Theorem \ref{ThmCEH}. A similar argument is also being considered in the context of hyperspaces \cite{Zava}. 
\end{Remark}

\begin{Remark}\label{RemWag}
The "finite determination" characterization of Hilbert space embeddability (Theorem \ref{ThmNow}) can be used to give another argument leading to the coarse  non-embeddability of   $\iDi$. For, if  $\iDi$ were coarsely embeddable in a Hilbert space,  one could take finite subsets of any metric space $Y$ which is not coarsely embeddable in $\ell_2$ (for example any $\ell_p$ with $p > 2$, see \cite{JohnRan}) and (by Lemma  \ref{LemiDi}) considering those  to be finite subsets of  $\iDi$  get the contradictory conclusion that $Y$ is coarsely embeddable in $\ell_2$. Using the same "finite determination" characterization Wagner (\cite{Wag}) showed that  $\pDi$ is not coarsely embeddable in a Hilbert space for $p > 2$, by first showing that  $(\mathbb{R}^N,\lVert \cdot \rVert_p)$ (for arbitrary $N$) isometrically embeds in $\pDi$ and then using a similar argument as above. The following theorem and corollary summarizes the discussion in this remark.
\end{Remark}

\begin{Theorem}\label{ThmFinDet}
Let $(X,d)$ be a metric space that coarsely embeds in a Hilbert space. Let $(Y,D)$ be a metric space whose finite subsets uniformly coarsely embed in $X$, i.e. for $i=1,2$ there are non-decreasing functions $\rho_i:[0,\infty) \to [0, \infty)$  with  $\lim_{t \to \infty} \rho_1(t) = \infty$, such that for every finite subset $A \subset Y$ there exists a map $f_A: A \to X$ satisfying 
$\rho_1(D(y_1,y_2)) \le d( f_A(y_1),f_A(y_2)) \le \rho_2(D(y_1,y_2)) $ for all $y_1,y_2 \in Y$. Then $Y$ coarsely embeds in a Hilbert space.

\begin{Corollary}\label{CorSumFinDet}
\begin{enumerate}
\item Let $(X,d)$ be a metric space such that every finite metric space isometrically embeds in $(X,d)$. Then $(X,d)$ does not coarsely embed in a Hilbert space. In particular, $\iDi$ does not coarsely embed in Hilbert space.
\item Finite subsets of $\ell_p$ uniformly coarsely embed in  $\pDi$. Therefore,  (Wagner, \cite{Wag} Theorem 10) $\pDi$ does not coarsely embed in a Hilbert space for $p>2$.
\end{enumerate}
\end{Corollary}

\end{Theorem}





\end{document}